\documentclass[11pt, a4paper]{amsart}
\usepackage[english]{babel}
\usepackage{fancyhdr}
\usepackage{latexsym, amssymb, amsmath, amscd, amscd, amsthm, amsxtra}
\usepackage{arydshln}
\usepackage [latin1]{inputenc}
 \usepackage{graphicx}
\usepackage[all]{xy}
\usepackage{placeins}
\usepackage{setspace}
\usepackage{hyperref}

\usepackage{appendix}
\usepackage{subfigure}

 \usepackage{indentfirst}
\newcommand{\PP}{\mathbf P}

\DeclareMathOperator{\Spec}{Spec}
\DeclareMathOperator{\Proj}{Proj}

\theoremstyle{plain}
\newtheorem{prop}{Proposition}[section]

\newtheorem{cor}[prop]{Corollary}
\newtheorem{thm}[prop]{Theorem}

\newtheorem{conjecture}[prop]{Conjecture}

\newtheorem*{prop*}{Proposition}
\newtheorem*{lem*}{Lemma}
\newtheorem*{sublem*}{Sublemma}
\newtheorem*{cor*}{Corollaire}
\newtheorem*{thm*}{Theorem}
\newtheorem*{hypo*}{Hypothesis}
\newtheorem*{question*}{Question}
\newtheorem*{conjecture*}{Conjecture}
\newtheorem*{scholum*}{Scholum}

\newtheorem*{defn*}{Definition}

\theoremstyle{slanted}

\newtheorem*{ex*}{Example}
\newtheorem*{exs*}{Examples}

\newtheorem*{rmk*}{Remark}
\newtheorem*{rmks*}{Remarks}

\newtheorem*{notation*}{Notation}

\theoremstyle{definition}

\newtheorem*{con*}{Construction}
\newtheorem{rmk}[prop]{Remark} 
\theoremstyle{remark}

\newtheorem*{warning*}{Warning}
\newtheorem*{shortnote*}{Note}
\newtheorem*{claim*}{Claim}
\newtheorem*{axiom*}{Axiom}
\newtheorem*{ack}{Acknowledgment}

\begin{document}

\title{Syzygies of general projections of canonical and paracanonical curves}

\author{Li Li}
\thanks{The author was funded by the Deutsche Forschungsgemeinschaft (DFG, German Research Foundation) under Germany's Excellence Strategy -- The Berlin Mathematics Research Center MATH$+$ (EXC-2046/1, project ID: 390685689) when submitting the paper. During the revision, he moved to the Institute for Basic Science and is now supported by the Institute for Basic Science (IBS-R032-D1)}
\date{}

\newcommand{\Addresses}{{
  \bigskip
  \footnotesize

  \textsc{\newline
  	Institute for Basic Science, Center for Complex Geometry\\
  	55 Expo-ro, Yuseong-gu\\
  	 34126 Daejeon\\ 
  	 Republic of Korea}\par\nopagebreak
  \textit{E-mail address}, \texttt{lili@ibs.re.kr}

}}

\maketitle

\begin{abstract}
Let $X\subset \mathbb{P}^{r}$ be an integral linearly normal variety, and
let $R=k[x_{0},\dots ,x_{r}]$ be the coordinate ring of
$\mathbb{P}^{r}$. It is known that the syzygies of $X$ contain some geometric
information. In recent years the syzygies of nonprojectively normal varieties
or, in other words, the projection~$X'$ of $X$ away from a linear subspace
$W\subset \mathbb{P}^{r}$, were taken into considerations. Assuming that
the coordinate ring of the ambient space that $X'$ lives in is $S$, there
are two types of vanishing properties of the Betti diagrams of the projected
varieties, the so-called $N_{d,p}^{S}$ and $\widetilde{N}_{d,p}$. The former
has been widely discussed for general varieties, for example, by S. Kwak,
Y. Choi, and E. Park, whereas the latter was discussed by W. Lee and E.
Park for curves of very large degree.

In this paper, I discuss the properties $\widetilde{N}_{d,p}$ of the projection
of a generic canonical and paracanonical curve away from a generic point
and in particular whether they are cut out by quadrics. Some conjectures
will be claimed based on the tests on $\emph{Macaulay2}$.
\end{abstract}

\section{Introduction}
%%LEAP%%%\label{sec1}
\label{S1}

Let $X$ be a projective variety embedded in the projective space
$\mathbb{P}^{r}=\mathbb{P}H^{0}(X,L)^{\vee}$ over an~algebraically closed
field $k$ via the complete linear system of a very ample line bundle
$L$. Let $R=k[x_{0},\dots ,x_{r}]$ be the coordinate ring of
$\mathbb{P}^{r}$. For the $R$-module
$R(X)=\bigoplus_{t\geq 0}H^{0}(X,\mathcal{O}_{X}(t))$, there is
a minimal free resolution
\begin{equation*}
	\cdots \to R_{2}\to R_{1}\to R_{0}=R\to
	R(X)\to 0,
\end{equation*}
where $R_{i}=\bigoplus_{j\geq 1}R(-i-j)^{b_{i,j}}$. It is well
known (for example, see \cite{ES05} and \cite{AN}) that
$b_{i,j}=\dim \operatorname{Tor}_{i}(R(X),k)_{i+j}$, where the space
$\operatorname{Tor}_{i}(R(X),k)_{i+j}$ can be identified as the Koszul cohomology
\begin{equation*}
	\bigwedge^{i+1}H^{0}(X,L)\otimes
	R(X)_{j-1}\to \bigwedge^{i}H^{0}(X,L)
	\otimes R(X)_{j}\to \bigwedge^{i-1}H^{0}(X,L)
	\otimes R(X)_{j+1}
\end{equation*}
with the morphism
$\bigwedge^{i}H^{0}(X,L)\otimes R(X)_{j}\to \bigwedge^{i-1}H^{0}(X,L)\otimes R(X)_{j+1}$ for arbitrary $i$, $j$ defined
as
\begin{equation*}
	v_{1}\wedge \cdots \wedge v_{i}\otimes P\mapsto \sum
	_{t=1}^{i}(-1)^{t}v_{1}
	\wedge \cdots \wedge \widehat{v_{t}}\wedge \cdots \wedge
	v_{i} \otimes v_{t}P.
\end{equation*}
The $\operatorname{Tor}$ group $\operatorname{Tor}_{i}(R(X),k)_{i+j}$ is also called
the Koszul cohomology group, denoted by $K_{i,j}(X,L)$. We say that
$X$ satisfies $N_{d,m}$ if $b_{i,j}=0$ for $1\leq i\leq m$ and
$j\geq d$ and $N_{m}$ for short if $d=2$.

Assume that $V\subset H^{0}(X,L)$ is a generic hyperplane, which induces
a projection
\begin{equation*}
	\pi _{p}:\mathbb{P}H^{0}(X,L)^{\vee}
	\dashrightarrow \mathbb{P}V^{\vee}
\end{equation*}
away from a point $p\in \mathbb{P}^{r}$. Denote by
$S:=\operatorname{Sym}^{*}V$ the coordinate ring of $\mathbb{P}V$. The intersection
ideal $I(X'):=I(X)\cap S$ defines the projection $\pi _{p}:X\to X'$, where
the notation of restriction is omitted if no ambiguity is led to. By the
genericity of $V$ this means that we are interested in the outer projection,
that is, $p\notin X$, and $\pi _{p}$ is an isomorphism.

There are two ways of considering the syzygies of $X'$. First, consider
the $S$-module
\begin{equation*}
	M=\bigoplus_{t\geq 0}H^{0}(X',
	\mathcal{O}_{X'}(t)),
\end{equation*}
as in \cite{KP} and \cite{non-complete}. In this case, the Betti numbers
are $b_{i,j}^{S}=\dim \operatorname{Tor}_{i}(M,k)_{i+j}$, where
$\operatorname{Tor}_{i}(M,k)_{i+j}$ is computed as the Koszul cohomology
\begin{equation*}
	\bigwedge^{i+1}V\otimes M_{j-1}\to
	\bigwedge^{i}V \otimes M_{j}\to
	\bigwedge^{i-1}V\otimes M_{j+1},
\end{equation*}
where the differentials are defined similarly as in the linearly normal
situation. We denote this Koszul cohomology by $K_{i,j}^{S}(X,L;V)$. We
say that $X'\subset \mathbb{P}(V)$ satisfies the property
$N_{d,m}^{S}$ if $b_{i,j}=0$ for $1\leq i\leq m$ and $j\geq d$. Correspondingly,
the property $N_{m}^{S}$ is satisfied if the first $m$ steps of the resolution
above are linear, meaning that $S_{i}=S(-i-1)^{b_{i,1}^{S}}$ for
$1\leq i\leq m$.

Another consideration is to look at the coordinate ring $S(X')$ of
$X'\subset \mathbb{P}(V)$. Then
\begin{equation*}
	S(X')=k\oplus V\oplus \biggl(\bigoplus
	_{t\geq 2}H^{0}(X',
	\mathcal{O}_{X'}(t))\biggr),
\end{equation*}
which is seen as an $S$-module. There is a free resolution of
$S(X')$ as follows:
\begin{equation*}
	\cdots \to \widetilde{S}_{2}\to \widetilde{S}_{1}\to S
	\to S(X')\to 0,
\end{equation*}
where
$\widetilde{S_{i}}=\bigoplus_{j\geq 1}S(-i-j)^{\widetilde{b}_{i,j}}$.
In this case,
$\widetilde{b}_{i,j}=\dim \operatorname{Tor}_{i}(S(X'),k)_{i+j}$, where the associated
Koszul cohomology is that of the differentials:
\begin{equation*}
	\bigwedge^{i+1}V\otimes S(X')_{j-1}
	\to \bigwedge^{i}V \otimes S(X')_{j}
	\to \bigwedge^{i-1}V\otimes S(X')_{j+1},
\end{equation*}
where the differentials are defined similarly to the previous. In this
case, we denote the Koszul cohomology group by
$\widetilde{K}_{i,j}(X,L;V)$. As defined in \cite{Birk}, the property
$\widetilde{N}_{d,m}$ means that $\widetilde{b}_{i,j}=0$ for
$1\leq i\leq m$ and $j\geq d$ and there is corresponding
$\widetilde{N}_{m}$ if $d=2$. Note that for $j=0$ and $j\geq 2$, we have
$S(X')_{j}=H^{0}(X',\mathcal{O}_{X'}(j))$, but for $j=1$,
$S(X')_{1}=V$, which is not a cohomological description. This discordance
is the main difficulty in computing the syzygies.

The normality of projected varieties was discussed by Alzati and Russo
\cite{AR}. Syzygy problems about $b_{i,j}^{S}$ were widely investigated
by Choi, Kwak, and Park \cite{KP,non-complete}. The problems about
$\widetilde{b}_{i,j}$ for curves of large degrees were solved by Lee and
Park \cite{LP}. In this paper, I focus on the general projections of canonical
curves and paracanonical curves for the $\widetilde{b}_{i,j}$ problems.
Notice that if the projection center $p$ is on the curve, then the problem
coincides with the inner projection, which was investigated in
\cite{CHOI2006859} and \cite{Han2008AnalysisOS}. Take the canonical curve
$C\subset \mathbb{P}^{g-1}$ with $g\geq 5$ as an example. Either inner
or general outer projected curve in~$\mathbb{P}^{g-2}$ is isomorphic to
$C$, because the line bundle $\omega _{C}(-p)$ is very ample for an arbitrary
$p\in C$ by \cite[Proposition~3.1]{Har}for the former case, and a general
point $p\in \mathbb{P}^{g-1}$ is not in the secant variety
$\Sigma _{1}(X)$ by the dimension counting by Lange \cite{L1} for the later.
However, we cannot solve the problem of general outer projection by deforming
to the case of inner projection, because the inner projection and the general
outer projection of the curve have different Hilbert polynomials, indicating
that such a family of curves is not flat (\cite[IX, Proposition~2.2]{ACG}).
Concretely speaking, the pullback of the hyperplane section to the inner
projected curve is $\omega _{C}(-p)$, whereas it is $\omega _{C}$ in the
outer projection case. The inner projections are avoided by taking the
general projection center.

For a generic canonical curve
$X=C\subset \mathbb{P}^{g-1}=\mathbb{P}H^{0}(C,\omega _{C})$ of genus
$g\geq 6$, we have the following basic results.

%t1.1 #&#
\begin{thm}%
	\label{thm1.1}%
	Let $C\subset \mathbb{P}^{g-1}$ be a generic canonical curve of genus
	$g\geq 6$. For a generic $p\in \mathbb{P}^{g-1}$, let $C'$ be the projection
	of $C$ away from $p$ to a hyperplane $H=\mathbb{P}V$. Then we have%

	\item[(1)] $\dim I_{C'}(2)=\frac{(g-1)(g-6)}{2}$;
		
	\item[(2)] $C'$ is $2$-normal, meaning that the multiplication map
		$\operatorname{Sym}^{2}V\to H^{0}(\omega _{C}^{2})$ is surjective.
	
\end{thm}

This is a special case of Theorem~2.7 in \cite{AR}. I will present the
proof in a slightly different way, because I care about how the quadric
generators look in further use.

By computing the Koszul cohomology, we have partial descriptions of the
Betti diagram of~$C'$.
%
%t1.2 #&#
\begin{thm}%
	\label{thm1.2}%
	With the assumptions in Theorem~\ref{thm1.1}, we have:%

		\item[(1)] The Castelnuovo regularity $\operatorname{reg}(C')$ is $3$. Concretely
		speaking, firstly, $\widetilde{b}_{i,q}=0$ for $q\geq 4$, and secondly,
		there is exactly one nonzero position in the third row of the Betti diagram
		of $C'$, which is $\widetilde{b}_{g-3,3}=1$.%
		
		\item[(2)] The length of the free resolution of $S(C')$ is exactly
		$g-2$. Moreover, we have
		\begin{equation*}
			\widetilde{b}_{p,2}=\begin{cases}
				0,& p\geq g-1,
				\\
				1,& p=g-2,
				\\
				g,& p=g-3.
			\end{cases}
		\end{equation*}
\end{thm}
This theorem shows that the Betti diagram of $C'$ looks like the following:%
\begin{equation*}
	\begin{tabular}{c|c c c c c}
		&$0$&$1$&$\cdots $&$g-3$&$g-2$\\
		\hline
		$0$&$1$& & & &\\
		$1$& &$\frac{(g-1)(g-6)}{2}$&$\cdots $& &\\
		$2$& &? &$\cdots $ &$g$&1\\
		$3$& & & &$1$ &
	\end{tabular}
\end{equation*}

We cannot decide whether the question mark is $0$ at this moment; we will
discuss this later. The length of the first row is also unknown. However,
it definitely does not exceed $\lfloor \frac{g}{2}\rfloor -1$, which is
the length of the first row of the Betti diagram of $C$ by \cite{v1} and
\cite{v2}. It shows that for a~generic canonical curve the Castelnuovo
regularity does not change after being projected into a codimension
$1$ subspace, whereas \cite{KP} and \cite{non-complete} predicted that
the Castelnuovo regularity may increase by $1$ for general varieties. For
$g=6$, since it is already known that $I_{C'}(2)=0$, we can obtain the
complete Betti diagram of $C'$:
\begin{equation*}
	\begin{tabular}{c|c c c c c}
		&0&1&2&3&4\\
		\hline
		0&1& & & &\\
		1& & & & &\\
		2& &10 &15 &6&1\\
		3& & & &1 &
	\end{tabular}
\end{equation*}

Denote the intersection of all the varieties defined as the vanishing locus
of quadrics in $I_{C'}(2)$ by $Y$, that is,
$Y=\bigcap_{Q\in I_{C'}(2)}V(Q)$. Clearly, $C'\subset Y$. Then
we show that in the case $g\geq 8$, for a generic point $q\in C'$, the
dimension of $Y$ locally at $q$ is $1$. Geometrically, this means that
$Y$ contains $C'$ as an irreducible component. This argument also shows
that for $g=7$, the projection $C'$ cannot be cut out only by quadrics,
although the accurate dimension of $I_{C'}(3)$ is still unknown. The story
in the genus $7$ gives a more intuitive evidence on the failure of deforming
to the inner projection. The syzygy property of the inner projection, for
example, Theorem~3.1 of \cite{CHOI2006859}, indicates that the inner projection
of a generic canonical curve of genus $7$ is cut out by quadrics, but its
generic outer projection fails to be.

For $g\geq 8$, by testing on $\emph{Macaulay2}$ \cite{m2} and using the
vector bundle method developed by Lange \cite{L1} and Lange and Sernesi
\cite{L2}, we can show the following:
%
%t1.3 #&#
\begin{thm}%
	\label{thm1.3}%
	Keeping the conventions as in previous theorems and assuming that
	$g\geq 8$, $C'$ is cut out by quadrics set-theoretically. When $g=8$, the
	natural multiplication
	\begin{equation*}
		I_{2}(V)\otimes V\to I_{3}(V)
	\end{equation*}
	is an isomorphism.
\end{thm}
%
%r1.4 #&#
\begin{rmk}%
	%%LEAP%%%\label{rem1.4}
	\label{rmk1.4}%
	The proposition for the case $g=8$ can be checked by finding an example
	via \emph{Macaulay2} and the openness of the condition being cut out by
	quadrics. In fact, consider the flat family ${X}$ of all the curves in
	$\mathbb{P}^{6}$ with the Hilbert polynomial $q(t)=14t-7$ over any scheme~$S$:%
	\begin{equation*}
		\hspace*{-24pt}\xymatrix{&\quad \quad \quad \quad{X}\ar@{->}[d]^\xi \subset \mathbb{P}^{6}\times S\\
			&S
		}%
	\end{equation*}

	Then there is a morphism
	$\phi :\xi _{*}\mathcal{I}_{X}(2)\otimes \xi _{*}\mathcal{O}_{
		\mathbb{P}^{6}\times S}(1)\to \xi _{*}\mathcal{I}_{X}(3)$ on $S$, where
	$\mathcal{I}_{X}$ is the ideal sheaf of $X$ in
	$\mathbb{P}^{6}\times S$, whose fiber at $s\in S$ is exactly
	$\phi _{s}:I_{X_{s}}(2)\otimes H^{0}(\mathcal{O}_{\mathbb{P}^{6}}(1))
	\to I_{X_{s}}(3)$. The surjectivity of this morphism on fibers is an open
	condition in $S$. In particular, if $S$ is all the projections of canonical
	curves of genus $8$, as a component of the Hilbert scheme
	$\operatorname{Hilb}_{6}^{q(t)}$, by the Castelnuovo regularity of a generic
	object of $S$, the surjectivity of $\phi _{s}$ ensures that $X_{s}$ is
	cut out by quadrics.\looseness=1
	
	Consider the locus
	$\mathcal{H}^{\mathrm{ns}}_{14,8,6}=\{[C\subset \mathbb{P}^{6},L]\in
	\operatorname{Hilb}_{6}^{q(t)}\mid C \text{ is smooth, and}$ $H^{1}(C,L)=0\}$ of smooth and nonspecial curves. Then the
	locus $S$ of projected canonical curves can be seen as the boundary of
	$\mathcal{H}^{\mathrm{ns}}_{14,8,6}$, and, in particular, they are in the same component
	of $\operatorname{Hilb}_{6}^{q(t)}$. As indicated in \cite{Ve}, the forgetting
	map $[C\subset \mathbb{P}^{6},L]\mapsto [C,L]$ to the universal Picard
	variety $\mathcal{P}\mathit{ic}^{14}_{8}$ is dominant. On the other hand, it was
	shown in \cite{genus8} that the failure of being cut out by quadrics is
	divisorial on $\mathcal{P}\mathit{ic}^{14}_{8}$, and hence it is pulled back to
	a divisor on $S$ via the composition
	$S\hookrightarrow \overline{\mathcal{H}^{\mathrm{ns}}}\dasharrow \mathcal{P}\mathit{ic}^{14}_{8}$.
	We conclude that the failure of being cut out by quadrics for projected
	canonical curves of genus $8$ is divisorial. Therefore the genus $8$ part
	of Theorem~\ref{thm1.3} is shown by finding a concrete example.
\end{rmk}
Combining what was done in \cite{KP} and \cite{non-complete} and the results
obtained by testing on $\emph{Macaulay2}$, we can compare the Betti diagrams
of (generic projections of) generic canonical curves in different senses:
\begin{equation*}
	\begin{footnotesize}
		\begin{tabular}{|c|c| c| c|}
			\hline
			&$b_{i,j}$ &$b_{i,j}^{S}$&$\widetilde{b}_{i,j}$\\
			\hline
			vanishing of the second row&$N_{\mathrm{Cliff}-1}$ &$N_{\mathrm{Cliff}-2}^{S}$ &$\widetilde{N}_{\mathrm{Cliff}-3}$ (expected, $g\geq 9$)\\
			\hline
			length of linear strands&$g-\mathrm{Cliff}-2$ &$g-\mathrm{Cliff}-2$ &$g-\mathrm{Cliff}-3$ (expected, $g\geq 9$)\\
			\hline
		\end{tabular}%
	\end{footnotesize}
\end{equation*}

Recall that a paracanonical curve is a curve equipped with a paracanonical
bundle $L:=\omega _{C}\otimes \eta $, where
$\eta \in \operatorname{Pic}^{0}(C)\setminus \{\mathcal{O}_{C}\}$ is a generic
nontrivial torsion line bundle. We say that $\eta $ is of level $l$ if
$\eta ^{\otimes l}\cong \mathcal{O}_{C}$ and that $C$ is Prym-canonical
if $l=2$. The genericity assumption on $\eta $ ensures that
$\eta \neq C_{2}-C_{2}$, that is, $\eta $ is not the difference of two
effective divisors of degree $2$, so that the paracanonical line bundle
$\omega _{C}\otimes \eta $ is very ample. Applying the methods used for
canonical curves to paracanonical curves, we have the following:%
%
%t1.5 #&#
\begin{thm}%
	\label{thm1.5}%
	Let $C$ be a generic paracanonical curve of genus $g\geq 11$ with a generic
	nontrivial torsion line bundle $\eta $ and a generic point
	$p\in \mathbb{P}^{g-2}=\mathbb{P}H^{0}(\omega _{C}\otimes \eta )^{
		\vee}$. Then the projection $C'$ of~$C$ away from $p$ is cut out by quadrics set-theoretically.
\end{thm}
In Section~\ref{S2}, we will discuss the normality of $C'$ and some basic
information on the Betti diagrams of $C'$, which proves Theorems~\ref{thm1.1}
and \ref{thm1.2}. In Section~\ref{S3}, we will explain the
failure of $C'$ being cut out by quadrics for $g=7$ and show that for
$g\geq 8$, the intersection of all the quadrics passing through $C'$ contains
it as an irreducible component. Section~\ref{S4} contributes
to the proof of the cases $g\geq 9$ of Theorem~\ref{thm1.3} using the Lange--Sernesi vector bundle method. Section~\ref{S5}
deals with paracanonical curves, and we will give the proof of Theorem~\ref{thm1.5}
under the assumption $g\geq 13$. In Section~\ref{S6}, we exhibit
the testing result on $\emph{Macaulay2}$, based on which we will claim some
conjectures on the property $\widetilde{N}_{m}$ of $C'$ and the length
of the linear strands of the resolutions and give the complementary parts
of the proofs of Theorems~\ref{thm1.3} and \ref{thm1.5}.

\begin{ack}
	I thank my supervisor Prof. Gavril Farkas for introducing me this nice topic, in which there are many interesting relations between the geometric aspects and the syzygy resolutions to be observed and investigated, and in particular, for his reminding me of the divisoriality of the failure of being cut out by quadrics. I'm grateful to Jieao Song for his assistance on the programming and discussions related to the Hilbert schemes and to Andr\'{e}s Rojas for his patient guidance to understand Lange's vector bundle methods. I thank Daniele Agostini for his discussions with me on the very ampleness of the paracanonical bundles. I thank Prof. Francesco Russo for indicating the mistake in the previous version during his visit to Daejeon in March 2026. I thank the referee for the improvement of the contents of the paper, especially for the advice to explain why the deformation to the inner projection fails.
\end{ack}

%s2 #&#
\section{Basics on Quadrics Containing Projected Canonical Curves}
%%LEAP%%%\label{sec2}
\label{S2}

We start from calculating the dimension of $I_{C'}(2)$ and show the
$2$-normality of $C'$ from the dimension counting.
%
%t2.1 #&#
\begin{thm}%
	\label{thm2.1}%
	Let $C\subset \mathbb{P}^{g-1}$ be a generic canonical curve of genus
	$g\geq 6$. For a generic $p\in \mathbb{P}^{g-1}$, let $C'$ be the
	projection of $C$ away from $p$ to a hyperplane $H=\mathbb{P}V$. Then
	\begin{equation*}
		\dim I_{C'}(2)=\frac{(g-1)(g-6)}{2},
	\end{equation*}
	and $C'$ is $2$-normal.
\end{thm}
\begin{proof}
	By the genericity of $C$ we know from, for example, \cite{Sch} and
	\cite{ACGH} that $C$ is $l$-normal for all $l\geq 2$ and, in particular,
	$\dim I_{C}(2)=\frac{(g-2)(g-3)}{2}$. Without loss of generality, we may
	assume that $p=[1:0:\cdots :0]$ and $V=\{x_{0}=0\}$. Then the projection
	$\pi _{p}:\mathbb{P}^{g-1}\dashrightarrow H$ is simply dropping the
	$0$th coordinate
	$[x_{0}:x_{1}:\cdots :x_{g-1}]\mapsto [x_{1}:\cdots :x_{g-1}]$. Let
	$R=k[x_{0},\dots ,x_{g-1}]$ and $S=k[x_{1},\dots ,x_{g-1}]$ be the coordinate
	rings of $\mathbb{P}^{g-1}$ and $\mathbb{P}V$, respectively. Hence there
	is a natural embedding $S\hookrightarrow R$ and
	$I_{C'}(2)=I_{C}(2)\cap S$. Assume that $I_{C}(2)$ has a basis
	$\{F_{1},\dots ,F_{\frac{(g-2)(g-3)}{2}}\}$, where $F_{i}$ are quadric
	polynomials in $R$, and we have the matrix
	\begin{equation*}
		F=  \left(
		\begin{matrix}
			F_{1}
			\\
			\vdots
			\\
			F_{\frac{(g-2)(g-3)}{2}}
		\end{matrix}
		\right).
	\end{equation*}
	If the Jacobian matrix $\operatorname{Jac}(F)$ is of full rank at $p$, then
	after applying elementary transformations, we may assume that
	$\operatorname{Jac}(F)(p)=  \big(
	\begin{smallmatrix}
		I_{g}
		\\
		0
	\end{smallmatrix}
	\big)$. Hence
	\begin{equation*}
		F=   \left(
		\begin{matrix}
			\frac{1}{2}x_{0}^{2}+G_{1}
			\\
			x_{0}x_{1}+G_{2}
			\\
			\vdots
			\\
			x_{0}x_{g-1}+G_{g}
			\\
			G_{g+1}
			\\
			\vdots
			\\
			G_{\frac{(g-2)(g-3)}{2}}
		\end{matrix}
		\right),
	\end{equation*}
	where $G_{i}$ are quadrics in~$S$. Then
	$\{G_{g+1},\dots ,G_{\frac{(g-2)(g-3)}{2}}\}$ form a basis of
	$I_{C'}(2)$, and the dimension is
	$\dim I_{C}(2)-g=\frac{(g-1)(g-6)}{2}$. The condition required is that
	$\operatorname{Jac}(F)$ is of full rank at $p$. We need to show that such a choice of
	$p$ is generic. In fact, the complement of the locus where $\operatorname{Jac}(F)$ is
	of full rank is a closed subset of $\mathbb{P}^{g-1}$, which is the union
	of the vanishing locus of each $(g-1)$-minor of $\operatorname{Jac}(F)$. For any fixed
	set of $g-1$ rows of $\operatorname{Jac}(F)$, these rows are linearly independent because
	the original homogeneous polynomials are linearly independent. Therefore
	the vanishing locus of this $(g-1)$-minor is of codimension $1$. Since
	there are finitely many $(g-1)$-minors, the union of these loci is still
	of codimension $1$. Summarizing the discussion above, we can choose~$p$
	in the open subset
	$\{p\in \mathbb{P}^{g-1}\mid \operatorname{rank}\operatorname{Jac}(F)(p)=g\}$ and let
	$p=[1:0:\cdots :0]$ after applying a coordinate transformation. Note that
	the assumption $g\geq 6$ ensures that there are enough linearly independent
	quadrics in $I_{C}(2)$, meaning that
	$\dim I_{C}(2)=\frac{(g-2)(g-3)}{2}\geq g$.
	
	We turn to show the $2$-normality. Note that there is an exact sequence
	of sheaves on $\mathbb{P}^{g-2}=\mathbb{P}V$:
	\begin{equation*}
		0\to \mathcal{I}_{C'}(2)\to \mathcal{O}_{\mathbb{P}^{g-2}}(2)\to
		\mathcal{O}_{C'}(2)=\omega _{C}^{2}\to 0.
	\end{equation*}
	Taking the global sections, we have an exact sequence of vector spaces:
	\begin{equation*}
		0\to I_{C'}(2)\to \operatorname{Sym}^{2}V\to
		H^{0}(\omega _{C}^{2}).
	\end{equation*}
	Then we have
	\begin{align*}
		\dim \operatorname{Im}(\operatorname{Sym}^{2}V\to H^{0}(\omega
		_{C}^{2}))&=\dim \operatorname{Sym}^{2}V-\dim
		I_{C'}(2)
		\\
		&=\binom{2+(g-1)-1}{g-1-1}-\frac{(g-1)(g-6)}{2}
		\\
		&=3g-3.
	\end{align*}
	This coincides with the dimension of $H^{0}(\omega _{C}^{2})$. Hence the
	morphism $\operatorname{Sym}^{2}V\to H^{0}(\omega _{C}^{2})$ is surjective, meaning
	that $C'$ is $2$-normal.
\end{proof}
We can obtain some information about the Betti diagram of
$C'\hookrightarrow \mathbb{P}^{g-2}$ by computing the Koszul cohomology.
%
%t2.2 #&#
\begin{thm}%
	\label{thm2.2}%
	With the assumptions and conclusions in Theorem~\ref{thm2.1}, we have the
	following facts about the Betti numbers $\widetilde{b}_{i,j}$:%
	
%	\begin{enumerate}[(3)]
		\item[(1)] $\widetilde{b}_{i,j}=0$ for $j\geq 4$;%
		
		\item[(2)] $\widetilde{b}_{g-3,3}=1$, and $\widetilde{b}_{i,3}=0$ for
		$i\neq g-3$;%
		
		\item[(3)] $\widetilde{b}_{i,2}=0$ for $i\geq g-1$,
		$\widetilde{b}_{g-2,2}=1$, and $\widetilde{b}_{g-3,2}=g$.
%	\end{enumerate}
\end{thm}
\begin{proof}%
	First, we consider the case $j\geq 3$. Since
	$C'\hookrightarrow \mathbb{P}V$ is a closed embedding, we know that the
	linear system corresponding to $V\subset H^{0}(\omega _{C})$ is generated
	by global sections. Hence there is an exact sequence of sheaves on
	$C$:
	\begin{align}
		\label{seq1} \tag{$*$} 0\to M_{V}\to V\otimes
		\mathcal{O}_{C}\to \omega _{C}\to 0.
	\end{align}
	Then we know from Remark~2.7 of \cite{AN} that for $j\geq 3$,
	\begin{equation*}
		\widetilde{K}_{i,j}(C,L;V)=\operatorname{coker} \biggl(\bigwedge
		^{i+1}V \otimes H^{0}(\omega _{C}^{j-1})
		\to H^{0}\biggl(\bigwedge^{i}M_{V}
		\otimes \omega _{C}^{j}\biggr) \biggr).
	\end{equation*}
	Taking the duality and using Serre's duality, we have
	\begin{align*}
		\widetilde{K}_{i,j}(C,L;V)^{\vee}& = \ker
		\biggl(H^{1}\biggl(\bigwedge^{i}M_{V}^{\vee}
		\otimes \omega _{C}^{1-j}\biggr)\to \bigwedge
		^{i+1}V^{\vee}\otimes H^{1}(\omega
		_{C}^{2-j}) \biggr)
		\\
		& = \ker \biggl(H^{1}\biggl(\bigwedge^{g-2-i}M_{V}
		\otimes \omega _{C}^{2-j}\biggr) \to \bigwedge
		^{g-2-i}V\otimes H^{1}(\omega _{C}^{2-j})
		\biggr).
	\end{align*}
	Since there is an exact sequence
	\begin{equation*}
		0\to \bigwedge^{g-2-i}M_{V}\otimes
		\omega _{C}^{2-j}\to \bigwedge
		^{g-2-i}V\otimes \omega _{C}^{2-j}\to
		\bigwedge^{g-3-i}M_{V}\otimes \omega
		_{C}^{3-j}\to 0
	\end{equation*}
	and the induced long exact sequence on cohomology, we have
	\begin{align*}
		\widetilde{K}_{i,j}(C,L;V)^{\vee}& = \operatorname{coker} \biggl(
		\bigwedge^{g-2-i}V\otimes H^{0}(\omega
		_{C}^{2-j})\to H^{0}\biggl(\bigwedge
		^{g-3-i}M_{V}\otimes \omega _{C}^{3-j}
		\biggr) \biggr)
		\\
		& = H^{0}\biggl(\bigwedge^{g-3-i}M_{V}
		\otimes \omega _{C}^{3-j}\biggr)
		\\
		& = \ker \biggl(\bigwedge^{g-3-i}V\otimes
		H^{0}(\omega _{C}^{3-j}) \to \bigwedge
		^{g-4-i}V\otimes H^{0}(\omega
		_{C}^{4-j}) \biggr).
	\end{align*}
	Clearly, $H^{0}(\omega _{C}^{4-j})=0$ for $j\geq 4$. Statement (1) is proved.
	
	Now assume that $j=3$. If $i=g-3$, then
	$\widetilde{K}_{i,j}(C,L;V)^{\vee}\cong H^{0}(\mathcal{O}_{C})$. Hence
	$\widetilde{b}_{g-3,3}=1$. Obviously, $\widetilde{b}_{i,3}=0$ if
	$i>g-3$. For $i<g-3$,
	$\widetilde{K}_{i,3}(C,L;V)^{\vee}\cong \ker  (\bigwedge^{g-3-i}V
	\to \bigwedge^{g-2-i}V\otimes H^{0}(\omega _{C}) )=0$, since
	the Koszul differential
	$\bigwedge^{g-3-i}V\to \bigwedge^{g-2-i}V\otimes V$ is
	injective as the starting morphism of the Koszul resolution of the field
	$k$.
	
	Now assume that $j=2$. In this case the Koszul cohomology groups become
	\begin{align*}
		\widetilde{K}_{i,2}(C,L;V)& = \operatorname{coker} \biggl(\bigwedge
		^{i+1}V \otimes V\to H^{0}\biggl(\bigwedge
		^{i}M_{V}\otimes \omega
		_{C}^{2}\biggr) \biggr),
		\\
		\widetilde{K}_{i,2}(C,L;V)^{\vee}& = \ker
		\biggl(H^{1}\biggl(\bigwedge^{i}M_{V}^{\vee}
		\otimes \omega _{C}^{-1}\biggr)\to \bigwedge
		^{i+1}V^{\vee}\otimes V^{\vee} \biggr)
		\\
		& = \ker \biggl(H^{1}\biggl(\bigwedge^{g-2-i}M_{V}
		\biggr)\to \bigwedge^{g-2-i}V\otimes V^{\vee}
		\biggr).
	\end{align*}
	Obviously, $\widetilde{b}_{i,2}=0$ for $i\geq g-1$. For $i=g-2$, it is
	the isomorphism
	\begin{equation*}
		\widetilde{K}_{g-2,2}(C,L;V)^{\vee}\cong \ker
		(H^{1}(\mathcal{O}_{C}) \to V^{\vee} ).
	\end{equation*}
	Hence $\widetilde{b}_{g-2,2}=1$. For $i=g-3$, we have
	\begin{equation*}
		\widetilde{K}_{g-3,2}(C,L;V)^{\vee}\cong \ker
		(H^{1}(M_{V})\to V \otimes V^{\vee} ).
	\end{equation*}
	Note that there is a commutative diagram with the exact horizontal row:
	\begin{equation*}
		\resizebox{324pt}{!}{\xymatrix{
				0\ar[r] &V\ar[r] &H^{0}(\omega _{C})\ar[r] &H^{1}(M_{V})\ar[r] &V\otimes H^{0}(\omega _{C})^{\vee}\ar[r]^{\quad \quad \quad \mathrm{ev}}\ar[d]^\phi &k\ar[r] &0\\
				& & & &V\otimes V^{\vee}\ar@{-->}[ru]_{\mathrm{ev}'}
		}}
	\end{equation*}
	where $\phi $ is the restriction map
	$v\otimes f\mapsto v\otimes f|_{V}$. The composition
	$H^{1}(M_{V})\to V\otimes V^{\vee}\to k$ is $0$, since the first row is
	the long exact sequence of the cohomology of \eqref{seq1}. Now assume that
	$\alpha \in \ker (V\otimes V^{\vee}\to k)$. We will show that
	$\alpha $ comes from $H^{1}(M_{V})$. Indeed, if we write
	$\alpha =\sum_{i=1}^{N}v_{i}\otimes f_{i}$ with
	$v_{i}\in V$, $f_{i}\in V^{\vee}$, then
	$\alpha =\phi (\sum_{i=1}^{N}v_{i}\otimes \widetilde{f}_{i})$,
	where $\widetilde{f}_{i}$ is the extension of $f_{i}$ to
	$H^{0}(\omega _{C})$ by defining $\widetilde{f}_{i}(x_{0})=0$. Denote
	$\sum_{i=1}^{N}v_{i}\otimes \widetilde{f}_{i}=:\beta $. Then
	$\operatorname{ev}(\beta )=0$. By the exactness of the horizontal row, $\beta $ comes
	from $H^{1}(M_{V})$. We conclude that
	$H^{1}(M_{V})\to V\otimes V^{\vee}\to k\to 0$ is exact. Therefore
	$\operatorname{Im} (H^{1}(M_{V})\to V\otimes V^{\vee} )=\ker (V
	\otimes V^{\vee}\to k)$, whose dimension is $(g-1)^{2}-1$. By the dimension
	counting of the exact horizontal row, $\dim H^{1}(M_{V})=(g-1)g$. Therefore
	$\dim \ker  (H^{1}(M_{V})\to V\otimes V^{\vee} )=g(g-1)-[(g-1)^{2}-1]=g$.
	\text{This completes the proof of (3).}
\end{proof}
Since we already know that $\widetilde{b}_{1,1}=0$ for $g=6$, it is possible
to give the complete Betti diagram of $C'$.
%
%c2.3 #&#
\begin{cor}%
	\label{cor2.3}%
	Let $C$ be a generic canonical curve of genus $6$, and let $C'$ be as before.
	Then the Betti diagram of $C'$ is
	\begin{equation*}
		\begin{tabular}{c|c c c c c}
			&0&1&2&3&4\\
			\hline
			0&1& & & &\\
			1& & & & &\\
			2& &10 &15 &6&1\\
			3& & & &1 &
		\end{tabular}
	\end{equation*}
\end{cor}
\begin{proof}
	It is already known that $C'$ is $2$-normal and hence $l$-normal for all
	$l\geq 2$. Hence we have
	$\widetilde{b}_{1,2}=\dim S^{3}V-\dim H^{0}(\omega _{C}^{3})=
	\binom{3+5-1}{5-2}-(5\times 6-5)=10$.%
	
	Now consider the differentials
	\begin{equation*}
		\resizebox{324pt}{!}{\xymatrix{
				0\ar[r]&\bigwedge^{4}V\ar[r]^{\phi _{1}}&\bigwedge^{3}V\otimes V\ar[r]^{\phi _{2}}&\bigwedge^{2}V\otimes H^{0}(\omega _{C}^{2})\ar[r]^{\phi _{3}} &V\otimes H^{0}(\omega _{C}^{3})\ar[r]^{\phi _{4}} &H^{0}(\omega _{C}^{4})\ar[r]&0.
		}}
	\end{equation*}
	By the $4$-normality, $\phi _{4}$ is surjective. So
	$\dim \ker (\phi _{4})=5\cdot (5g-5)-(7g-7)=90$. Since
	$\widetilde{b}_{1,3}$ vanishes,
	$\dim \operatorname{Im}(\phi _{3})=\dim \ker (\phi _{4})=90$. Therefore%
	\begin{equation*}
		\begin{aligned}
			\dim \ker (\phi _{3})& = \dim \bigwedge^{2}V\otimes H^{0}(
			\omega _{C}^{2})-\dim \operatorname{Im}(\phi _{3})
			\\
			& = \binom{5}{2}\cdot (3g-3)-90
			\\
			& = 60.
		\end{aligned}
	\end{equation*}
	
	On the other hand, by the injectivity of $\phi _{1}$ and the vanishing
	of $\widetilde{b}_{3,1}$,
	\begin{equation*}
		\dim \ker (\phi _{2})=\dim \operatorname{Im}(\phi _{1})=
		\binom{5}{4}=5.
	\end{equation*}
	Therefore
	\begin{equation*}
		\dim \operatorname{Im}(\phi _{2})=\bigwedge^{3}V
		\otimes V-\dim \ker (\phi _{2})=\binom{5}{3}\cdot 5-5=45.
	\end{equation*}
	Since $\widetilde{b}_{3,1}=0$, $\phi _{1}$ is injective. Then
	$\widetilde{b}_{2,2}=\dim \ker (\phi _{3})-\dim \operatorname{Im}(\phi _{2})=60-45=15$.
	We draw the required conclusion by combining this with Theorem~\ref{thm2.2}.
\end{proof}
%s3 #&#
\section{Analysis on Quadric Generators}
%%LEAP%%%\label{sec3}
\label{S3}

We are going to investigate the quadric generators of the ideal of the
projected canonical curve. The goal is to give an evidence that for a generic canonical
curve $C\subset \mathbb{P}^{g-1}$ of genus $g\geq 8$ its
projection away from a general point $p\in \mathbb{P}^{g-1}$ is cut out
by quadrics.

Let $C\subset \mathbb{P}^{g-1}$ be a generic canonical curve with
$g\geq 8$. Consider a generic hyperplane
$H\cong \mathbb{P}^{g-2}\subset \mathbb{P}^{g-1}$. We know that $H$ intersects
with $C$ in $2g-2$ in general position by Chapter III, Section~1 of
\cite{ACGH}. Take a general point
$p_{0}\in \mathbb{P}^{g-1}\setminus C$ as the point that the curve is projected
away from. In particular, such a point is chosen to be outside $H$. Without
loss of generality, denote $p_{0}=[1:0:\cdots :0]$ and
$H=\{x_{0}=0\}$ as before and take $g-1$ points on $H$, denoted by
$p_{1}=[0:1:0:\cdots :0],\dots ,p_{g-1}=[0:\cdots :0:1]$. Take the sections
$x_{0},x_{1},\dots ,x_{g-1}\in \mathbb{P}^{g-1}$ such that
$x_{i}(p_{j})=\delta _{ij}$ for $0\leq i,j\leq g-1$.

Take $g=8$ as an example. It has been seen that the quadric generators
of $C$ are spanned by $8$ quadrics containing $x_{0}$, say
$F_{i}=x_{0}x_{i}+P_{i}(x_{1},\dots ,x_{7})$ for $0\leq i\leq 7$, and
$7$ quadrics in $x_{1},\dots ,x_{7}$, namely
$F_{j}=P_{j}(x_{1},\dots ,x_{7})$ for $7\leq j\leq 14$. If the projected
curve $C'$ were cut out by quadrics, after dropping the quadrics with
$x_{0}$, the Jacobian matrix of the remaining terms would have rank
$6-1=5$ at the points $p_{1},\dots ,p_{7}$ by the smoothness of $C'$. On
the other hand, the total rank of the coefficients of $x_{7}$ of all the
terms, including those with $x_{0}$, should be $7-1=6$ by the smoothness
of $C$. This means that the summands with $x_{7}$ appearing in
$F_{0},\dots ,F_{6}$ can be eliminated by those from
$F_{8},\dots ,F_{14}$ after applying some elementary row transformations,
whereas $F_{7}$, since it contains $x_{0}x_{7}$, would survive from the
eliminations from $F_{8},\dots ,F_{14}$. Therefore by applying some elementary
row transformations we may without loss of generality assume that the quadric
generators of $C$ are spanned by quadrics of the following types:
\begin{equation*}
	\begin{cases}
		F_{i}=x_{0}x_{i}+P_{i}(x_{1},\dots ,x_{6}),&0\leq i\leq 6,
		\\
		F_{7}=l_{0}(x_{0},x_{1},\dots ,x_{6})x_{7}+P_{7}(x_{1},\dots ,x_{6}),&
		\\
		F_{j}=l_{j}(x_{1},\dots ,x_{6})x_{7}+Q_{j}(x_{1},\dots ,x_{6}),&8
		\leq j\leq 12,
		\\
		F_{13}=R_{1}(x_{1},\dots ,x_{6}),F_{14}=R_{2}(x_{1},\dots ,x_{6}),&
	\end{cases}
\end{equation*}
where $P_{*}$, $Q_{*}$, and $R_{*}$ are quadrics in
$x_{1},\dots ,x_{6}$, $l_{0}$ is a linear form in
$x_{0},x_{1},\dots ,x_{6}$ with the nonzero coefficient of $x_{0}$, and
$l_{i}$ are linear forms in $x_{1},\dots ,x_{6}$ for $1\leq i\leq 5$. Due
to the vanishing orders at $p_{7}$, we may further find that all
$l_{i}$, $0\leq i\leq 5$, are sections in
$H^{0}(C,\mathcal{O}(K-2p_{7}))$. Moreover, by the geometric Riemann--Roch
theorem the dimension of $H^{0}(C,\mathcal{O}(K-2p_{7}))$ is exactly
$6$, which in other words means that $l_{0},\dots ,l_{5}$ are exactly the
basis of $H^{0}(C,\mathcal{O}(K-2p_{7}))$. Such a story would be verified
by checking that after deleting the basis functions with $x_{7}$, there
are exactly $7$ remaining linearly independent terms with $x_{0}$, namely
$F_{0},\dots ,F_{6}$. This means that the Jacobian matrix of the quadrics
in the ideal of the inner projection $C''$ away from $p_{7}$ to the hyperplane
$\{x_{7}=0\}$, which passes through $p_{0}$, has full rank at
$p_{0}$. This motivates the following theorem.
%
%t3.1 #&#
\begin{thm}%
	\label{thm3.1}%
	Let $C\subset \mathbb{P}^{g-1}$ be a canonical curve of genus
	$g\geq 8$. Let $p\in \mathbb{P}^{g-1}\setminus C$ be a~general point, and
	let $H$ be a hyperplane not passing through $p$. Denote by $C'$ the projection
	$C$ away from $p$ to $H$. Let $S$ be the set of all quadrics on $H$ containing
	$C'$ and denote the intersection $\bigcap_{Q\in S}Q=:Y$. Then
	$Y$ contains $C'$ as an irreducible component.
\end{thm}
\begin{proof}
	For $q\in C$, we can find that
	$p_{0}\in \mathbb{P}H^{0}(C,\omega _{C}(-q))^{\vee}$. Since for a generic
	curve of genus at least $8$, its gonality is at least $5$, the line bundle
	$\omega _{C}(-q)$ is very ample by \cite[Proposition~3.1(b)]{Har}. This
	implies that the image $C''$ of the inner projection defined by the linear
	system $ (\omega _{C}(-q),H^{0} (\omega _{C}(-q) ) )$ is an
	embedding of $C$ into $H'=\mathbb{P}H^{0}(C,\omega _{C}(-q))^{\vee}$. Replacing
	$\omega _{C}$ by $\omega _{C}(-q)$ in Theorem~\ref{thm2.1}, which requires
	the condition $\dim I_{2}(C)-(g-2)\geq g-1$, implying that $g\geq 8$, we
	find that for a general point $y\in H'$, the rank of the Jacobian matrix
	of a basis of $I_{2}(\omega _{C}(-q))$ at $y$ is $g-1$, where
	$I_{2}(\omega _{C}(-q))$ is the kernel of the morphism
	\begin{equation*}
		S^{2}H^{0}(C,\omega _{C}(-q))\to
		H^{0}(C,\omega _{C}(-q)^{\otimes 2}).
	\end{equation*}
	For a generic point $y=[y_{0}:\cdots :y_{g-1}]\in \mathbb{P}^{g-1}$, let
	the hyperplane $H(y)\subset \mathbb{P}^{g-1}$ be its dual hyperplane defined
	by $\sum_{i=0}^{g-1}y_{i}x_{i}=0$. Consider the set of pairs
	$P=\{(y,q)\in \mathbb{P}^{g-1}\times C\mid q\in H(y)\cap C\}$ and its subset
	$Q= \{(y,q)\in P |\operatorname{rank}\operatorname{Jac} (I_{2} (\omega _{C}(-q)
	) )(y)\leq g-2 \}\subset P$, where
	$\operatorname{Jac} (I_{2} (\omega _{C}(-q) ) )$ is the Jacobian matrix of
	a basis of $I_{2}(\omega _{C}(-q))$. Let
	$\pi _{1}:P\to \mathbb{P}^{g-1}$ and $\pi _{2}:P\to C$ be the projections
	to the first and second components, respectively. Since the rank of the
	Jacobian matrix $\operatorname{Jac} (I_{2} (\omega _{C}(-q) ) )$ is
	$g-1$ generically on $\mathbb{P}H^{0}(\omega _{C}(-q))^{\vee}$, we deduce
	that for a point $q\in C$, the locus
	$\{y\in \mathbb{P}H^{0}(\omega _{C}(-q))^{\vee}\mid (y,q)\in Q\}$ has dimension
	at most $g-3$. Therefore the total dimension of $\pi _{1}(Q)$ has dimension
	at most $g-2$ as $q$ traveling through $C$. In other words, for a general
	$p_{0}\in \mathbb{P}^{g-1}$, in the sense that $p_{0}$ is in the complement
	of $\pi _{1}(Q)$, there exists a preimage in $(p_{0},q)\in P$ such that
	$\operatorname{rank}\operatorname{Jac} (I_{2} (\omega _{C}(-q) ) )(p_{0})=g-1$.
	
	Applying a coordinate transformation, we can write
	$p_{0}=[1:0:\cdots :0]$. Assume that $H=\{x_{0}=0\}$ is the dual hyperplane
	of $p_{0}$ and it intersects with $C$ at $2g-2$ points
	$p_{1},\dots ,p_{2g-2}$ in general position. Assume that
	$p_{i}=[0:\cdots :0:1:0:\cdots :0]$ with $1$ appearing in $i$th position
	and other coordinates $0$ for $1\leq i\leq g-1$. We may let
	$q=p_{g-1}=[0:\cdots :0:1]$. Then the sections
	$x_{0},\dots ,x_{g-1}\in H^{0}(C,\omega _{C})$ satisfy
	$x_{i}(p_{j})=\delta _{ij}$ for $0\leq i,j\leq g-1$. Summarizing, we know
	that the quadric generators of $I(C)$ are of the following forms:
	\begin{align*}
		F_{0}& = x_{0}^{2}+P_{0}(x_{1},
		\dots ,x_{g-2}),
		\\
		F_{1} & =  x_{0}x_{1}+P_{1}(x_{1},
		\dots ,x_{g-2}),
		\\
		&\ \ \vdots
		\\
		F_{g-2}& = x_{0}x_{g-2}+P_{g-2}(x_{1},
		\dots ,x_{g-2}),
		\\
		F_{g-1}& =  l_{0}(x_{0},x_{1},
		\dots ,x_{g-2})x_{g-1}+Q_{0}(x_{1},
		\dots ,x_{g-2}),
		\\
		F_{g}& =  l_{1}(x_{1},\dots
		,x_{g-2})x_{g-1}+Q_{1}(x_{1},
		\dots ,x_{g-2}),
		\\
		&\ \ \vdots
		\\
		F_{2g-4}& = l_{g-3}(x_{1},\dots
		,x_{g-2})x_{g-1}+Q_{g-3}(x_{1},
		\dots ,x_{g-2}),
		\\
		F_{2g-3}& = R_{1}(x_{1},\dots
		,x_{g-2}),
		\\
		&\ \ \vdots
		\\
		F_{\frac{(g-1)(g-4)}{2}}& = R_{\frac{g^{2}-9g+12}{2}}(x_{1},\dots
		,x_{g-2}),
	\end{align*}
	where $P_{*}$, $Q_{*}$, and $R_{*}$ are quadrics in
	$x_{1},\dots ,x_{g-2}$, $l_{0}$ is a linear form in
	$x_{0},x_{1},\dots ,x_{g-2}$ with the nonzero coefficient of $x_{0}$, and
	$l_{i}$ are linearly independent linear forms in
	$x_{1},\dots ,x_{g-2}$ for $1\leq i\leq l_{g-3}$. In particular, at
	$p_{g-1}\in C'$, the rank of the Jacobian matrix of quadric generators
	is $g-3$. By the lower semicontinuity of the rank of a matrix, this rank
	function is at least $g-3$ at a generic point $q'$ on $C'$. On the other
	hand, if $Y=\bigcap_{Q\in I_{2}(C')}V(Q)$, then
	$Y\supset C'$, and for any point $q'\in C'$, we have
	$\operatorname{rank}\operatorname{Jac}(I_{2}(C'))(q')\leq g-2-\dim A(Y)_{q'}\leq g-3$, where
	$A(Y)$ is the coordinate ring of $Y$, and $A(Y)_{q'}$ is its localization
	at $q'$. Therefore $\operatorname{rank}\operatorname{Jac}(I_{2}(C'))(q')$ is exactly
	$g-3$ at a generic point $q'\in C'$, and, in particular,
	$\dim A(Y)_{q'}=1$. Then by \cite[Lemma~33.20.3]{stack} the irreducible
	component containing $C'$ must be of dimension $1$. Therefore $C'$ is an
	irreducible component.
\end{proof}
Note that for $g=7$, we deduce from Theorem~\ref{thm2.1} that
$\dim I_{2}(C')=3$. There is no chance that the three quadrics satisfy
the rank condition in the proof of Theorem~\ref{thm3.1}. Therefore we have
the following:
%
%c3.2 #&#
\begin{cor}%
	\label{cor3.2}%
	Let $C\hookrightarrow \mathbb{P}^{6}$ be a generic canonical curve of genus
	$7$. Then its projection $C'$ away from a generic point
	$p\in \mathbb{P}^{6}$ is not cut out by quadrics.
\end{cor}
Assume that $g\geq 8$. Theorem~\ref{thm3.1} gives us evidence that $C'$ is the intersection of all the quadrics containing it. The set-theoretic proof will be given in Section~\ref{S4}.

If $g=8$, then from the commutative diagram
\begin{equation*}
	\resizebox{324pt}{!}{\xymatrix{
			0\ar[r] &\bigwedge^{3}V\ar[r] &\bigwedge^{2}V\otimes V\ar[r] &V\otimes H^{0}(2\omega _{C})\ar[r] &H^{0}(3\omega _{C})\ar[r]&0\\
			0\ar[r] &\bigwedge^{3}V\ar[r] &\bigwedge^{2}V\otimes V\ar@{=}[u]\ar[r] &V\otimes S^{2}V\ar[u]\ar[r] &S^{3}V\ar[u]\ar[r]&0\\
			& &0\ar[u]\ar[r] &V\otimes I_{2}(V)\ar[u]\ar[r] &I_{3}(V)\ar[u]\ar[r] &0\\
	}}%
\end{equation*}
and the exactness of the middle row we have that
$\widetilde{b}_{1,2}=\widetilde{b}_{2,1}=\dim \operatorname{coker}(V
\otimes I_{2}(V)\to I_{3}(V))$. By Theorem~\ref{thm2.1} we have that
$\dim I_{2}(V)=7$, and by the $3$-normality of $C'$ we have
\begin{align*}
	\dim I_{3}(V)& = \dim S^{3}V-\dim H^{0}(3
	\omega _{C})
	\\*
	& = 49
	\\*
	& = \dim V\otimes I_{2}(V).
\end{align*}
In conclusion, if $C'$ were cut out by quadrics, then
$\widetilde{b}_{1,2}=\widetilde{b}_{2,1}=0$ and the multiplication map
$V\otimes I_{2}(V)\to I_{3}(V)$ would be an isomorphism. We summarize in
the following:
%
%t3.4 #&#
\begin{thm}%
	\label{thm3.4}
	Let $C\hookrightarrow \mathbb{P}^{g-1}$ be a generic canonical curve of
	genus $g\geq 8$. Then its projection~$C'$ away from a generic point
	$p\in \mathbb{P}^{g-1}$ is cut out by quadrics set-theoretically. If
	$g=8$, then $\widetilde{b}_{1,2}=\widetilde{b}_{2,1}=0$, and the multiplication
	map $V\otimes I_{2}(V)\to I_{3}(V)$ is an isomorphism.
\end{thm}
The proof of the theorem will be discussed in Sections \ref{S4} and
\ref{S6}. In Section~\ref{S6}, we will see more conjectures on the length
of the linear strands of the syzygy resolution and the property $\widetilde{N}_{p}$.

%s4 #&#
\section{Lange--Sernesi Vector Bundle Method}
%%LEAP%%%\label{sec4}
\label{S4}

In this section, we give explanations on the generic projection of a generic
canonical curve cut out by quadrics set-theoretically for $g\geq 9$.

As before, assume that $C\subset \mathbb{P}^{g-1}$ with $g\geq 9$ is a
generic canonical curve. Let $p_{0}=[1:0:\cdots :0]$ be a general
point away from which the curve is projected to the dual hyperplane
$H=\{x_{0}=0\}$. Let $V=\operatorname{span}\{x_{1},\dots ,x_{g-1}\}$, where
$\{x_{0},\dots ,x_{g-1}\}$ is a basis of $H^{0}(\omega _{C})$. Let
\begin{equation*}
	\pi _{p_{0}}:\mathbb{P}^{g-1}\dashrightarrow H,
	[x_{0}:x_{1}:\cdots :x_{g-1}] \mapsto
	[x_{1}:\cdots :x_{g-1}]
\end{equation*}
be the projection onto $H$. There is a restriction
$\pi _{p_{0}}|_{C}:C\to C'\subset H$ on $C$. Let
$Y=\bigcap_{Q\in I_{C'}(2)}V(Q)$ be the intersection of
all quadrics on $H$ containing $C'$. Clearly, $Y\supset C'$. Assume that
$Y\neq C'$. Then there exists $p\in Y\setminus C'$. Without loss of generality,
we may assume that $p=[1:0:\cdots :0]\in H$, where the points on $H$ are
represented in the coordinates on $H$.

We have the commutative diagram%
\begin{equation*}
	\xymatrix{
		&V\otimes V\ar[rd]^\mu \ar[d]_s & &\\
		&S^{2}V \ar[r]^\rho &H^{0}(2\omega _{C}) \ar[r]^{\quad \quad \mathrm{ev}_{p}} &k,
	}%
\end{equation*}
where $\operatorname{ev}_{p}$ is defined as the evaluation of the representative in
$S^{2}V$ of a quadric form at $p$. It is well-defined because
$p\in Y$, which implies that $\operatorname{ev}_{p}(I_{C'}(2))=0$. Note that
\begin{equation*}
	\eta :=\operatorname{ev}_{p}\in H^{0}(2\omega
	_{C})^{\vee}\cong \operatorname{Ext}^{1}( \omega
	_{C},\mathcal{O}_{C}).
\end{equation*}
Hence there is a rank $2$ vector bundle $E$ fitting into the exact sequence
%
%e1 #&#
\begin{align}
	0\to \mathcal{O}_{C}\to E\to \omega _{C}\to 0.
	\label{X1} %%LEAP%%%\label{eq1}
\end{align}
Since we have the commutative diagram%
\begin{equation*}
	\xymatrix{
		V\otimes V\ar@{^{(}->}[r]^-i \ar[rd]_\mu & H^{0}(\omega _{C})\otimes H^{0}(\omega _{C}) \ar[d]^m\\
		&H^{0}(2\omega _{C}),
	}%
\end{equation*}
$\eta m$ defines a morphism
$H^{0}(\omega _{C})\otimes H^{0}(\omega _{C})\to k$, which corresponds
to $\delta _{\eta}:H^{0}(\omega _{C})\to H^{0}(\omega _{C})^{\vee}$, exactly
the connection map of the induced long exact sequence of \eqref{X1} on
cohomology.

Note that $\eta $ defines an element $[\eta ]$ in
$\mathbb{P}H^{0}(2\omega _{C})^{\vee}$. Assume that
$\phi _{|2\omega _{C}|}:C\hookrightarrow \mathbb{P}H^{0}(2\omega _{C})^{
	\vee}$ is the embedding of $C$ via the complete linear system
$|2\omega _{C}|$. We can find a secant variety
$\Sigma _{d-1}(C)\subset \mathbb{P}H^{0}(2\omega _{C})^{\vee}$ such that
$[\eta ]\in \Sigma _{d-1}(C)$ if $d\geq \frac{3g-3}{2}$ by dimension counting.
Here I follow the conventions on secant varieties as in \cite{ENP}, meaning
that $\Sigma _{0}(C)=C$, which is different from Lange's convention. In
particular, $[\eta ]$ is in the span
$\langle D\rangle _{\phi _{|2\omega _{C}|}}$ of an effective divisor~$D$
of degree $d$ under the map $\phi _{|2\omega _{C}|}$. Thus we have
\begin{equation*}
	\eta \in \ker (H^{0}(2\omega _{C})^{\vee}\to
	H^{0}(2\omega _{C}-D)^{
		\vee}).
\end{equation*}
Write $A:=\omega _{C}(-D)$. Then by \cite{L1} there is an exact sequence
\begin{equation*}
	0\to A\to E\to D\to 0.
\end{equation*}
We are going to estimate $h^{0}(E)$ from which the cohomological information
of $A$ and $D$ would be deduced via analysis on the quadric generators
of $I_{C}$.

From the discussions in Sections~\ref{S2} and \ref{S3} it follows that
the quadric generators of $I_{C}$ are of the forms
\begin{align*}
	F_{0}& = x_{0}^{2}+P_{0}(x_{1},
	\dots ,x_{g-1}),
	\\
	F_{1} & =  x_{0}x_{1}+P_{1}(x_{1},
	\dots ,x_{g-1}),
	\\
	&\ \ \vdots
	\\
	F_{g-1}& =  x_{0}x_{g-1}+P_{g-1}(x_{1},
	\dots ,x_{g-1}),
	\\
	F_{g}& =  P_{g}(x_{1},\dots
	,x_{g-1}),
	\\
	&\ \ \vdots
	\\
	F_{N}& = P_{N}(x_{1},\dots
	,x_{g-1}),
\end{align*}
where $N=\frac{(g-1)(g-4)}{2}$, $F_{g},\dots ,F_{N}$ are quadric generators
of $I_{C'}$, and $V(F_{g})\cap \cdots \cap V(F_{N})=Y\ni p$. So no
$x_{1}^{2}$ appear in $F_{g},\dots ,F_{N}$. On the other hand, since
$p\notin C'$, there must be $x_{1}^{2}$ \text{appearing} in
$F_{0},\dots ,F_{g-1}$. By applying a coordinate change among
$x_{2},\dots ,x_{g-1}$, we may rewrite $F_{0},\dots ,F_{g-1}$ in the forms
\begin{align*}
	F_{0}& = x_{0}^{2}+a_{0}x_{1}^{2}+P_{0}(x_{1},
	\dots ,x_{g-1}),
	\\
	F_{1} & =  x_{0}x_{1}+a_{1}x_{1}^{2}+P_{1}(x_{1},
	\dots ,x_{g-1}),
	\\
	F_{2} & =  x_{0}x_{2}+a_{2}x_{1}^{2}+P_{2}(x_{1},
	\dots ,x_{g-1}),
	\\
	F_{3} & =  x_{0}x_{3}+P_{3}(x_{1},
	\dots ,x_{g-1}),
	\\
	&\ \ \vdots
	\\
	F_{g-1}& =  x_{0}x_{g-1}+P_{g-1}(x_{1},
	\dots ,x_{g-1}),
	\\
	F_{g}& =  P_{g}(x_{1},\dots
	,x_{g-1}),
	\\
	&\ \ \vdots
	\\
	F_{N}& = P_{N}(x_{1},\dots
	,x_{g-1}).
\end{align*}

Observe that
$x_{3},\dots ,x_{g-1}\in \ker (\delta _{\eta}:H^{0}(\omega _{C})\to H^{0}(
\omega _{C})^{\vee})$, because for $3\leq i\leq g-1$,
$\delta _{\eta}(x_{i})(x_{0})=\eta (x_{i}x_{0})=\eta (-P_{i})=-P_{i}(p)=0$,
whereas $\delta _{\eta}(x_{i})(x_{j})=0$ is simply verified since
$x_{i}x_{j}(p)$ for $j\neq0$ by the definition of $\eta $. We conclude
that $\dim \ker (\delta _{\eta})\geq g-3$.

We have obtained two exact sequences
%
%e2 #&#
%e3 #&#
\begin{align}
	0&\to \mathcal{O}_{C}\to E\to \omega _{C}\to 0,
	\label{X2} %%LEAP%%%\label{eq2}
	\\
	0&\to A\to E\to D\to 0. \label{X3} %%LEAP%%%\label{eq3}
\end{align}
Sequence \eqref{X2} implies that
$h^{0}(\mathcal{O}_{C})+h^{0}(\omega _{C})=h^{0}(E)+\dim \operatorname{Im}(
\delta _{\eta})$. We have
$\dim \operatorname{Im}(\delta _{\eta})\leq 3$ since
$\dim \ker (\delta _{\eta})\geq g-3$. Hence $h^{0}(E)\geq g-2$. Similarly,
from \eqref{X3} it follows that
\begin{equation*}
	h^{0}(A)+h^{0}(D)\geq h^{0}(E)\geq g-2.
\end{equation*}
Combining this with the Riemann--Roch theorem, we see that
$\operatorname{Cliff}(A)=\operatorname{Cliff}(D)\leq 3$, which means that
\begin{align*}
	\deg (D)-2h^{0}(D)+2&\leq 3,%
	\\
	\deg (A)-2h^{0}(A)+2&\leq 3.%
\end{align*}
Recall that such $D$ always exists if $\deg (D)\geq \frac{3g-3}{2}$. In
particular, we can choose
$\deg (D)=\lceil \frac{3g-3}{2}\rceil $, in which case
$\deg (A)=\lfloor \frac{g-1}{2}\rfloor $. So we have
\begin{equation*}
	h^{0}(D)\geq \frac{\deg (D)-1}{2}\geq \frac{\frac{3g-3}{2}-1}{2}=
	\frac{3g-5}{4}
\end{equation*}
and, similarly,
\begin{equation*}
	h^{0}(A)\geq \frac{\deg (A)-1}{2}>\frac{\frac{g-3}{2}-1}{2}=
	\frac{g-5}{4}.
\end{equation*}
It is easy to find that for $g\geq 9$, both $h^{0}(D)$ and
$h^{0}(A)$ are at least $2$. Then both $D$ and $A$ contribute to the Clifford
index. However, for a generic curve $C$, we have
$\operatorname{Cliff}(C)\geq \lfloor \frac{g-1}{2}\rfloor \geq 4$, whereas
$\operatorname{Cliff}(A)=\operatorname{Cliff}(D)\leq 3$, a contradiction. The discussions
above prove the main theorem in this section.
%
%t4.1 #&#
\begin{thm}%
	\label{thm4.1}
	Assume that $g\geq 9$. Let $C\hookrightarrow \mathbb{P}^{g-1}$ be a generic
	canonical curve of genus $g$, and let $p_{0}\in \mathbb{P}^{g-1}$ be a
	generic point. Then the projection $C'$ of $C$ away from $p_{0}$ is cut
	out by quadrics set-theoretically.
\end{thm}
%
%r4.2 #&#
\begin{rmk}%
	%%LEAP%%%\label{rem4.2}
	\label{rmk4.2}
	The technique developed above fails to explain the case of $g=8$, as
	$D$ can be chosen to be of degree 11, and $A$ to be of degree $3$ with
	$h^{0}(A)=1$, in which case there is no contradiction. The choice of
	$\deg (D)$ was the safest but far from the best.
\end{rmk}
%

%s5 #&#
\section{Paracanonical Curves}
%%LEAP%%%\label{sec5}
\label{S5}

In this section, we investigate the generic projection of a generic paracanonical
curve. Recall that a paracanonical curve of genus $g$ is embedded in
$\mathbb{P}^{g-2}$ via the line bundle $\omega _{C}\otimes \eta $, where
$\eta $ is a generic nontrivial $l$-torsion in $\operatorname{Pic}^{0}(C)$. Note
that by Theorem~1 of \cite{GL86} a generic paracanonical curve of genus
at least $7$ is projectively normal. Note that the line bundle
$\omega _{C}\otimes \eta $ is very ample if and only if
$\eta \notin C_{2}-C_{2}$ by \cite[IV, Proposition~3.1]{Har}. First, we
look at the case of odd genus $g$. It was proved in
\cite{FARKAS2003553} that for $g=2i+1$, the difference variety
$C_{i}-C_{i}$ is identified with
\begin{equation*}
	\{\xi \in \operatorname{Pic}^{0}(C)\mid K_{i,1}(C;\xi ,\omega
	_{C})\neq0\}
\end{equation*}
and further in \cite{CEFS} that with the same notations,
$K_{i,1}(C;\eta ,\omega _{C})=0$ for a general level $l$ paracanonical
curve $[C,\eta ]\in \mathcal{R}_{g,l}$. As a result,
$\eta \notin C_{2}-C_{2}$ for a general torsion line bundle $\eta $ of
level $l\geq 2$. For the case of even $g$, the locus of torsion bundles
is dense in the Jacobian of the curve, and therefore in general they are
not in $C_{2}-C_{2}$ for $g\geq 6$.

Following the proof of Theorem~\ref{thm2.1} or \cite[Theorem~2.7]{AR},
we can show the following:
%
%t5.1 #&#
\begin{thm}%
	\label{thm5.1}%
	Let
	$C\subset \mathbb{P}^{g-2}=\mathbb{P}H^{0}(\omega _{C}\otimes \eta )^{
		\vee}$, where $\eta $ is a generic nontrivial torsion line bundle, be a
	generic paracanonical curve of genus $g\geq 8$. For a generic
	$p\in \mathbb{P}^{g-1}$, let $C'$ be the projection of $C$ away from
	$p$ to a hyperplane $H=\mathbb{P}V$. Then
	$\dim I_{C'}(2)=\frac{(g-1)(g-8)}{2}$ and $C'$ is $2$-normal.
\end{thm}
%
%r5.2 #&#
\begin{rmk}%
	%%LEAP%%%\label{rem5.2}
	\label{rmk5.2}
	Similarly to the assumption $g\geq 6$ in Theorem~\ref{thm2.1}, the assumption
	$g\geq 8$ here ensures that there are enough linearly independent quadrics
	containing~$C$. Concretely speaking, the dimension of $I_{2}(C)$ is
	$\binom{g}{g-2}-(3g-3)=\frac{(g-1)(g-6)}{2}$. Hence the expected dimension
	of $I_{2}(C')$ is $\frac{(g-1)(g-6)}{2}-(g-1)=\frac{(g-1)(g-8)}{2}$. The
	expected dimension can be reached if and only if $g\geq 8$.
\end{rmk}

Hence we have the following elementary observations.
%
%p5.3 #&#
\begin{prop}%
	\label{prop5.3}
	Let $C\subset \mathbb{P}^{g-2}$ be a generic paracanonical curve of genus
	$g\geq 8$, and let $p\in \mathbb{P}^{g-2}$. Then the Betti diagram of the
	projection $C'$ away from $C$ satisfies:%
	%

%	\begin{enumerate}[(3)]
	\item[(1)] $\widetilde{b}_{i,j}=0$ for $j\geq 3$;%
		
	\item[(2)] $\widetilde{b}_{i,2}=0\iff i\geq g-2$ and
		$\widetilde{b}_{g-3,2}=1$;%
		
	\item[(3)] $\widetilde{b}_{g-4,2}=2g-2$.
%	\end{enumerate}
\end{prop}
\begin{proof}
	Similar to that of Theorem~\ref{thm2.2}.
\end{proof}
Applying the same method as in Section~\ref{S4}, we deduce the following:
%
%t5.4 #&#
\begin{thm}%
	\label{thm5.4}%
	Assume that $C\subset \mathbb{P}^{g-2}$ is a generic paracanonical curve
	of genus $g\geq 13$ with a~generic nontrivial torsion line bundle
	$\eta $, and let $p=[1:0:\cdots :0]\in \mathbb{P}^{g-2}$ be a generic point.
	Then the projection $C'$ of $C$ away from $p$ is cut out by quadrics set-theoretically.
\end{thm}
\begin{proof}
	As before, write $Y=\bigcap_{Q\in I_{C'}(2)}V(Q)$ for the intersection
	of all quadrics on $H$ containing~$C'$ and assume that $Y\neq C'$. Then
	there exists $p\in Y\setminus C'$. Without loss of generality, we may assume
	that $p=[1:0:\cdots :0]\in H=\{x_{0}=0\}=\mathbb{P}^{g-3}$.
	
	Then for $V=\{x_{0}=0\}$, we have the commutative diagram%
	\begin{equation*}
		\xymatrix{
			&V\otimes V\ar[rd]^\mu \ar[d]_s & &\\
			&S^{2}V \ar[r]^{\rho \quad \quad} &H^{0}(2(\omega _{C}\otimes \eta )) \ar[r]^{\quad \quad \mathrm{ev}_{p}} &k,
		}%
	\end{equation*}
	where $\operatorname{ev}_{p}$ is defined to be the evaluation of a representative in
	$S^{2}V$ at $p$. Note that
	\begin{equation*}
		\theta :=\operatorname{ev}_{p}\in H^{0}(2(\omega
		_{C}\otimes \eta ))^{\vee}\cong \operatorname{Ext}^{1}(
		\omega \otimes \eta ,\eta ^{\vee}),
	\end{equation*}
	which defines the exact sequence
	%
	%e4 #&#
	\begin{align}
		0\to \eta ^{\vee}\to E\to \omega \otimes \eta \to 0, \label{X4}
		%%LEAP%%%\label{eq4}
	\end{align}
	where $E$ is a rank $2$ vector bundle.
	
	On the other hand, we can see $[\theta ]$ as an element in
	$\mathbb{P}^{3g-4}=\mathbb{P}H^{0}(2(\omega _{C}\otimes \eta ))^{\vee}$.
	There exists a divisor $D$ on $C$ of degree $d$ such that
	$[\theta ]\in \langle D\rangle _{\phi _{|2(\omega _{C}\otimes \eta )|}}$
	if $d\geq \frac{3g-3}{2}$. Similarly as in Section~\ref{S4}, we can deduce
	the exact sequence
	%
	%e5 #&#
	\begin{align}
		0\to \omega \otimes \eta (-D)\to E\to \eta ^{\vee}(D)\to 0.
		\label{X5} %%LEAP%%%\label{eq5}
	\end{align}
	The connection map of \eqref{X4} is
	$\delta _{\theta}:H^{0}(\omega _{C}\otimes \eta )\to H^{0}(\omega _{C}
	\otimes \eta )^{\vee}$, whose kernel has dimension at least $g-4$. Hence
	\eqref{X4} implies that $h^{0}(E)\geq g-4$. The exact sequence
	\eqref{X5} implies that
	$h^{0}(\omega _{C}\otimes \eta (-D))+h^{0}(\eta ^{\vee}(D))\geq h^{0}(E)
	\geq g-4$. Then $\operatorname{Cliff}(D')=\operatorname{Cliff}(A')\leq 5$ for
	$D':=\eta ^{\vee}(D)$ and $A':=\omega _{C}(-D')$.
	
	Now we take $d=\lceil \frac{3g-3}{2}\rceil $. Then we obtain that
	$h^{0}(D')\geq \frac{\frac{3g-3}{2}-3}{2}\geq \frac{3g-9}{4}>2$ and
	$h^{0}(A')\geq \frac{\lfloor \frac{g-1}{2}\rfloor -3}{2}>
	\frac{g-9}{4}$, implying that $h^{0}(A')\geq 2$ when $g\geq 13$. Then
	$D'$ and $A'$ contribute to the Clifford index, but
	$\operatorname{Cliff}(C)=\lfloor \frac{g-1}{2}\rfloor \geq 6\geq
	\operatorname{Cliff}(D')$, a contradiction.
\end{proof}
This gives the proof of the $g\geq 13$ part of Theorem~\ref{thm1.5}. For
the cases $g=11,12$, we refer to Section~\ref{S6}.

%s6 #&#
\section{Testing on Macaulay2 and Open Problems}
%%LEAP%%%\label{sec6}
\label{S6}

In this section, I claim some conjectures on the syzygies of the projection
of a generic canonical curve, based on testing on \emph{Macaulay2}
\cite{m2}. As mentioned in \cite{CEFS}, the vanishing of Betti numbers
of a general genus $g$ curve can be verified via that of a $g$-nodal rational
curve over a finite field, which can be seen as the mod $p$ reduction of
a $g$-nodal rational curve over $\mathbb{Z}$. Since a $g$-nodal rational
curve is the limit of a family of smooth curves of genus $g$ in the moduli
space $\mathcal{M}_{g}$ of curves of genus $g$ and the dimension of the
Koszul cohomology is upper semicontinuous, the vanishing of the Betti numbers
of the smooth curves can be verified from that of the nodal curves. I use
the package \verb"NodalCurves.m2"\cite{NC} developed by Frank-Olaf Schreyer
and used in \cite{CEFS}. The codes are as follows:\vspace*{6pt}

\begin{verbatim}
	--Genus of the curve. I take g=8 as an example. Its value can be modified.
	g=8
	
	--Generate the ideal of a rational g-nodal curve.
	(L,I)=randomCanonicalNodalCurve(g)
	
	--The coordinate ring of the hyperplane P^{g-2} where the projected curve
	C' lies.
	R=L_0[t_1..t_(g-1)]
	
	--The coordinate ring of P^{g-1} where the curve C lies.
	S=L_5
	
	--The natural inclusion map from R to S, mapping t_i to t_i.
	F=map(S,R,{t_1..t_(g-1)})
	
	--The ideal of C'.
	J=preimage(F,I)
	
	--The Betti diagram of the coordinate ring of C'.
	betti res J
\end{verbatim}
\vspace*{6pt}
Compiling these codes, the result is as follows:
\begin{verbatim}
       0 1  2  3  4 5 6
total: 1 7 35 56 35 9 1
    0: 1 .  .  .  . . .
    1: . 7  .  .  . . .
    2: . . 35 56 35 8 1
    3: . .  .  .  . 1 .
\end{verbatim}
\vspace*{6pt}
Together with Remark~\ref{rmk1.4}, this verifies Theorem~\ref{thm3.4} for
$g=8$.

We can observe more information like the property
$\widetilde{N}_{p}$ and the length of the linear strands of the free resolution
of the ideal of the projected curve by doing more tests on higher genus
curves.

When $g=9$, the output is\vspace*{6pt}
\begin{verbatim}
       0  1  2  3  4   5  6 7
total: 1 12 46 96 100 48 10 1
    0: 1  .  .  .  .   .  . .
    1: . 12 16  .  .   .  . .
    2: .  . 30 96 100 48  9 1
    3: .  .  .  .  .   .  1 .
\end{verbatim}
%\vspace*{6pt}
When $g=11$, the output is\vspace*{6pt}
\begin{verbatim}
       0  1  2  3  4   5   6   7   8 9
total: 1 25 80 182 350 400 245 80 12 1
    0: 1  .  .  .   .   .   .   .  . .
    1: . 25 80  70  .   .   .   .  .
    2: .  .  . 112 350 400 245 80 11 1
    3: .  .  .  .   .   .   .   .  1 .
\end{verbatim}
\vspace*{6pt}
For $g=2k+1$, we can observe that firstly, the property
$\widetilde{N}_{k-3}$ holds and secondly, the length of the linear strands
is $k-2$, meaning that $\widetilde{b}_{p,1}\neq0$ if and only if
$1\leq p\leq k-2$.

As a comparison, we look at the case where the genus $g\geq 10$ is even.

When $g=10$, the output is\vspace*{3pt}
\begin{verbatim}
       0  1  2  3   4   5   6   7  8
total: 1 18 43 127 210 162 63  10  .
    0: 1  .  .  .   .   .   .   .  .
    1: . 18 42  1   .   .   .   .  .
    2: .  .  1 126 210 162 63  10  1
    3: .  .  .  .   .   .   .   1  .
\end{verbatim}
\vspace*{3pt}
When $g=12$, the output is\vspace*{3pt}
\begin{verbatim}
       0  1  2   3   4   5   6   7   8  9 10
total: 1 33 132 198 463 792 693 352 99 13  1
    0: 1  .  .   .   .   .   .   .   .  .  .
    1: . 33 132 198  1   .   .   .   .  .  .
    2: .  .  .   1  462 792 693 352 99 12  1
    3: .  .  .   .   .   .   .   .   .  1  .
\end{verbatim}
\vspace*{3pt}
We can observe that unlike the case of $g=8$, for even genus
$g=2k\geq 10$, we have
$\widetilde{b}_{k-2,1}=\widetilde{b}_{k-3,2}=1$. The curves satisfy
$\widetilde{N}_{k-4}$ but not $\widetilde{N}_{k-3}$.

Recall that the rank of a syzygy $\gamma \in K_{p,1}(X,L)$ is the dimension
of the unique linear subspace $W\subset H^{0}(X,L)$ of minimal dimension
such that $\gamma \in K_{p,1}^{S}(X,L;W)$. In our setting, for a syzygy
$\gamma \in \widetilde{K}_{p,1}(X,L;V)$, we consider the unique linear
subspace $W\subset V$ of minimal dimension such that $\gamma $ is represented
by an element in $\bigwedge^{p}W\otimes V$ and define the dimension
of $W$ as the rank of $\gamma $. The rank of the extra syzygy can be tested
as follows, taking $g=10$ as an example:\vspace*{3pt}
\begin{verbatim}
	--Linear strand of the ideal J of C'
	C=res(J,DegreeLimit=>1,LengthLimit=>3)
	
	--The coefficient matrix of the inclusion of the extra syzygy
	M=C.dd_3
	
	--The Jacobian matrix
	M=jacobian(transpose M)
	
	--rank of the syzygy
	rank M
\end{verbatim}
%\vspace*{6pt}
It turns out that when $g=10$, the output is $9$.\vadjust{\goodbreak}

Summarizing the discussions above, we can claim our conjectures.
%
%c6.1 #&#
\begin{conjecture}%
	\label{conj6.1}
	Let $C\hookrightarrow \mathbb{P}^{g-1}$ be a generic canonical curve of
	genus $g\geq 9$, and let $p\in \mathbb{P}^{g-1}$ be a generic point. Then
	the projection $C'$ away from $p$ to a hyperplane has the following properties:%
	
	%\begin{enumerate}[(3)]
		\item[(1)] $C'$ satisfies $\widetilde{N}_{\operatorname{Cliff}(C)-3}$;%
		
		\item[(2)] The length of the linear strands of the syzygy resolutions
		of $C'$ is $g-\operatorname{Cliff}(C)-3$, meaning that
		$\widetilde{b}_{t,1}\not=0$ if and only if
		$1\leq t\leq g-\operatorname{Cliff}(C)-3$;%
		
		\item[(3)] If $g=2k$, then
		$\widetilde{b}_{k-2,1}=\widetilde{b}_{k-3,2}=1$, and the extra syzygy is
		of maximal rank.
	%\end{enumerate}
\end{conjecture}
Still using the package \verb"NodalCurves.m2", the experiments about paracanonical
curves for any level can be done. The codes are as follows:\vspace*{6pt}
\begin{verbatim}
	--Ideal of a paracanonical curve C of genus g and level l
	(L,I)=randomPrymCanonicalNodalCurve(g,l)
	
	--Coordinate ring of the hyperplane where C' lies
	R=L_2[t_1..t_(g-2)]
	
	--Coordinate ring of Coordinate ring of the projective space where
	C lies
	S=L_7
	
	--The inclusion of R into S
	F=map(S,R,{t_1..t_(g-2)})
	
	--The ideal of C'
	J=preimage(F,I)
	
	--Betti diagram of the resolution of J
	betti res J
\end{verbatim}
\vspace*{6pt}
It turns out that no matter what $l$ is, the Betti diagrams look the same
for each genus. For example,\vspace*{6pt}

\noindent $g=8$:
\begin{verbatim}
       0  1  2  3  4 5
total: 1 21 49 42 14 1
    0: 1  .  .  .  . .
    1: .  .  .  .  . .
    2: . 21 49 42 14 1
\end{verbatim}
\vspace*{6pt}
$g=9$:
\begin{verbatim}
       0  1  2  3  4  5 6
total: 1 20 70 96 60 16 1
    0: 1  .  .  .  .  . .
    1: .  4  .  .  .  . .
    2: . 16 70 96 60 16 1
\end{verbatim}
%\eject
$g=10$:
\begin{verbatim}
       0  1  2  3   4   5  6 7
total: 1 12 81 171 165 81 18 1
    0: 1  .  .  .   .   .  . .
    1: .  9  .  .   .   .  . .
    2: .  3 81 171 165 81 18 1
\end{verbatim}

%r6.2 #&#
\begin{rmk}%
	%%LEAP%%%\label{rem6.2}
	\label{rmk6.2}%
	For $g=10$, there are enough quadrics containing $C'$, but $C'$ is still
	not cut out by quadrics. The numerical proof is not hard by calculating
	the difference $\widetilde{b}_{1,2}-\widetilde{b}_{2,1}$. Repeating the
	process of the proof of Theorem~\ref{thm3.1}, with the assumption
	$\eta \notin C_{3}-C_{3}$ generic, we can see that the intersections
	of the quadrics passing $C'$ contain $C'$ as an irreducible component.
	However, the pictures of other components are not clear.
\end{rmk}
For genus at least $11$, due to the limit of the CPU, I only show the first
two columns of the Betti diagrams.

For $g=11$, the first two steps of resolutions are\vspace*{6pt}
\begin{verbatim}
       0  1  2
total: 1 15 90
    0: 1  .  .
    1: . 15 20
    2: .  . 70
\end{verbatim}
\vspace*{6pt}
For $g=12$, the first two steps are\vspace*{6pt}
\begin{verbatim}
       0  1  2
total: 1 22 77
    0: 1  .  .
    1: . 22 55
    2: .  . 22
\end{verbatim}
\vspace*{6pt}
With these results, the cases of $g=11,12$ of Theorem~\ref{thm1.5} are
verified, and it is not hard to check that $\tilde{b}_{3,1}=0$ for
$g=11$ or $12$. We have the following conjecture.
%
%c6.3 #&#
\begin{conjecture}%
	\label{conj6.3}
	Assume that $C$ is a generic paracanonical curve of genus $g\geq 11$ with
	the paracanonical bundle $\omega _{C}\otimes \eta $, where $\eta $ is a
	generic nontrivial torsion line bundle. Let
	$p\in H^{0}(\omega _{C}\otimes \eta )^{\vee}$ be a generic point.
	Then the Betti diagram of the projection $C'$ of $C$ away from $p$ satisfies
	$\widetilde{N}_{\operatorname{Cliff}(C)-4}$, and the length of the linear strands
	is $\operatorname{Cliff}(C)-3$.
\end{conjecture}

\appendix
\section*{Appendix: A counter example to \v{S}okurov's Lemma}
The appendix is to clarify that a lemma cited in the previous version of this paper is wrong, and therefore the conclusions on this paper can only hold on the level of sets. 
\begin{lem*}[\v{S}okurov, 1971]
	%\label{lem3.3}
	Let $T$ be a nondegenerate projective variety in $\mathbb{P}^{N}$ generated
	by forms of degree not less than $n$. Assume that $S$ is the schematic
	intersection of the forms in~$I(T)_{n}$ and $S=T$ as topological spaces
	(i.e., $S_{\mathrm{red}}=T$). Then $S=T$ if and only if there is a~hyperplane
	$H$ such that $H\cap T$ is reduced and $H\cap T=H\cap S$.
\end{lem*}

In the original proof, the step
$$I(T\cap H)=I(T)+(l)$$
was wrong. As a result, we cannot deduce the equality on the scheme structure by simply comparing the hyperplane sections. \\

There is a counter example generated by ChatGPT Pro. Consider the morphism
\[
\nu:\PP^1\longrightarrow\PP^7,\qquad
[s:t]\longmapsto
[s^{10}:s^9t:s^8t^2:s^6t^4:s^3t^7:s^2t^8:st^9:t^{10}],
\]
and let \(Y=\nu(\PP^1)\). If \(S\) is the scheme cut out by all quadrics containing
\(Y\), then
\[
S_{\mathrm{red}}=Y,\qquad S\ne Y,
\]
whereas
\[
S\cap H=Y\cap H
\]
is reduced for the hyperplane \(H=\{x_0+x_7=0\}\).\\

The followings are the explanations to the counter example. They are not included in the paper.\\

The functions \(x_1/x_0=t/s\) and \(x_6/x_7=s/t\) give inverse
parameters on the two standard affine charts. The remaining
coordinates are powers of these parameters, so \(\nu\) is a closed
immersion. In particular, \(Y\) is a smooth rational curve of degree
\(10\). The coordinate monomials are linearly independent, hence
\(Y\) is nondegenerate and \(I(Y)_1=0\).

\medskip
\noindent\textbf{All quadratic equations.}
Let \(R=\mathbb{C}[x_0,\ldots,x_7]\). A basis of \(I(Y)_2\) is
\begin{align*}
	q_1&=x_0x_2-x_1^2,
	& q_2&=x_0x_3-x_2^2,\\
	q_3&=x_0x_5-x_1x_4,
	& q_4&=x_0x_5-x_3^2,\\
	q_5&=x_0x_6-x_1x_5,
	& q_6&=x_0x_6-x_2x_4,\\
	q_7&=x_0x_7-x_1x_6,
	& q_8&=x_0x_7-x_2x_5,\\
	q_9&=x_1x_7-x_2x_6,
	& q_{10}&=x_1x_7-x_3x_4,\\
	q_{11}&=x_2x_7-x_3x_5,
	& q_{12}&=x_3x_7-x_4^2,\\
	q_{13}&=x_4x_6-x_5^2,
	& q_{14}&=x_4x_7-x_5x_6,\\
	q_{15}&=x_5x_7-x_6^2.
\end{align*}
Indeed, the exponents of \(t\) in the parametrization are
\[
A=\{0,1,2,4,7,8,9,10\},
\]
and their pairwise sums fill \(\{0,\ldots,20\}\). Thus
\[
R_2\longrightarrow H^0(\mathbb{P}^1,\mathcal O_{\mathbb{P}^1}(20))
\]
is surjective, and
\[
\dim I(Y)_2=\binom{9}{2}-21=15.
\]
The displayed quadrics are linearly independent: within each group
of quadratic monomials with the same exponent sum, they are
differences of one fixed monomial and the others.

Set
\[
J=(q_1,\ldots,q_{15})=(I(Y)_2),\qquad S=\Proj(R/J).
\]

\medskip
\noindent\textbf{The embedded point on \(x_0=1\).}
Put \(a=x_1\) and \(u=x_4\). The equations
\(q_1,q_2,q_4,q_5,q_7\) give
\[
x_2=a^2,\qquad x_3=a^4,\qquad
x_5=a^8,\qquad x_6=a^9,\qquad x_7=a^{10}.
\]
After substitution, the remaining equations generate exactly
\[
\bigl(a(u-a^7),\,u^2-a^{14}\bigr)\subset\mathbb{C}[a,u].
\]
Writing \(\varepsilon=u-a^7\), we obtain
\[
S\cap D_+(x_0)
\simeq
\Spec\frac{\mathbb{C}[a,\varepsilon]}
{(a\varepsilon,\varepsilon^2)}.
\tag{1}
\]
The corresponding chart of \(Y\) is defined by \(\varepsilon=0\).
The class of \(\varepsilon\) is nonzero, since setting \(a=0\)
gives \(\mathbb{C}[\varepsilon]/(\varepsilon^2)\).
Moreover, its annihilator is \((a,\varepsilon)\), so it gives an
embedded associated point at
\[
p=[1:0:0:0:0:0:0:0].
\]
Away from \(p\), equation (1) agrees with the reduced curve.

For an explicit equation detecting this difference, take
\[
F=x_0^2x_4-x_1x_2x_3\in I(Y)_3.
\]
Its restriction to (1) is \(\varepsilon\ne0\). Thus \(F\) does not
vanish scheme-theoretically on \(S\).

\medskip
\noindent\textbf{The other chart and the support of \(S\).}
On \(x_7=1\), put \(v=x_6\). The equations
\(q_{15},q_{14},q_{12},q_{11},q_9,q_7\) give
\[
x_5=v^2,\quad x_4=v^3,\quad x_3=v^6,\quad
x_2=v^8,\quad x_1=v^9,\quad x_0=v^{10}.
\]
All other equations become identities. Hence
\[
S\cap D_+(x_7)\simeq\Spec\mathbb{C}[v],
\]
which is precisely the corresponding chart of \(Y\).

These two charts cover \(S\). In fact, if \(x_0=x_7=0\) at a point
of \(S\), then \(q_1,q_2,q_4,q_{12},q_{15}\) force
\[
x_1=x_2=x_3=x_4=x_6=0,
\]
and \(q_{13}\) then forces \(x_5=0\), which is impossible in
projective space. Consequently,
\[
S_{\mathrm{red}}=Y,\qquad S\ne Y,
\]
and the only difference is the embedded point at \(p\).

\medskip
\noindent\textbf{A reduced hyperplane section.}
Let \(H=\{x_0+x_7=0\}\). This hyperplane avoids \(p\).
Every point of \(S\cap H\) lies in \(D_+(x_0)\), since otherwise
\(x_0=x_7=0\). Using (1), we get
\[
S\cap H
\simeq
\Spec\frac{\mathbb{C}[a,\varepsilon]}
{(a\varepsilon,\varepsilon^2,1+a^{10})}.
\]
In this quotient \(a\) is invertible, so \(\varepsilon=0\). Therefore
\[
S\cap H=Y\cap H
\simeq
\Spec\frac{\mathbb{C}[a]}{(1+a^{10})}.
\]
The polynomial \(1+a^{10}\) has ten distinct roots over \(\mathbb{C}\),
so the common hyperplane section is reduced.
\bibliography{projection}{}
\bibliographystyle{alpha}
\Addresses

\end{document}